\documentclass[pdflatex,sn-mathphys-num]{sn-jnl}%

\usepackage{graphicx}%
\usepackage{multirow}%
\usepackage{amsmath,amssymb,amsfonts}%
\usepackage{amsthm}%
\usepackage{mathrsfs}%
\usepackage[title]{appendix}%
\usepackage{xcolor}%
\usepackage{textcomp}%
\usepackage{manyfoot}%
\usepackage{booktabs}%
\usepackage{algorithm}%
\usepackage{algorithmicx}%
\usepackage{algpseudocode}%
\usepackage{listings}%
\usepackage{epstopdf}

\theoremstyle{thmstyleone}%
\newtheorem{theorem}{Theorem}
\newtheorem{lemma}{Lemma}
\newtheorem{corollary}{Corollary}

\theoremstyle{thmstyletwo}%

\theoremstyle{thmstylethree}%

\begin{document}

\title[Parallel packing a square with isosceles right triangles and equilateral triangles]{Parallel packing a square with isosceles right triangles and equilateral triangles}

\author*[1,2]{\fnm{Chen-Yang} \sur{Su}}\email{suchenyang@tjnu.edu.cn}

\affil*[1]{\orgdiv{College of Mathematical Sciences}, \orgname{Tianjin Normal University}, \orgaddress{\street{Xiqing}, \postcode{300387}, \state{Tianjin}, \country{China}}}

\affil[2]{\orgdiv{Institute of Mathematics and Interdisciplinary Sciences}, \orgname{Tianjin Normal University}, \orgaddress{\street{Xiqing}, \postcode{300387}, \state{Tianjin}, \country{China}}}

\abstract{Suppose that $I$ is a unit square. Let $T$ (resp. $\Delta$) be an isosceles right triangle (resp. an equilateral triangle). We prove that any collection of triangles homothetic to $T$ (resp. $\Delta$), whose  total area does not exceed $\frac{1}{2}$ (resp. $\frac{\sqrt{3}}{4}$), can be parallel packed into $I$. These upper bounds are tight.}


%%\pacs[JEL Classification]{D8, H51}


\maketitle

\section{Introduction}\label{sec1}

Let $D, C_{1}, C_{2}, \ldots$ be plane convex bodies. We say that the collection $\{C_{n}\}$ permits a translative covering of $D$ if there exist translations $\tau_{n}$ such that $D\subset \bigcup\tau_{n}C_{n}$. We say that the convex bodies $C_{1}, C_{2}, \ldots$ are \emph{translatively packed} into $D$ if there exist translations $\tau_{n}$ such that $\tau_{n}C_{n}$ are subsets of $D$ and that they have pairwise disjoint interiors.
One side of a polygon $P$ is called the \emph{base} of $P$. A packing of a polygon $P$ with the polygons $P_{1}, P_{2}, \ldots$ is called \emph{parallel} if each $\tau_{n}P_{n}$ has a side parallel to the base of $P$.

Moon and Moser~\cite{J.W} showed that any collection of squares whose total area is not greater than $\frac{1}{2}$ can be
translatively packed into a unit square. Januszewski~\cite{J.J2} proved that any collection of squares of diagonals parallel to a unit square $I$ whose total area does not exceed $\frac{4}{9}$ can be translatively packed into $I$. Liu and Su~\cite{LS} considered the problem of packing a rhombus with squares.
About packing a triangle with its homothetic copies, one can see~\cite{J.J, RT, J.J0}. For parallel packing triangles with squares, one can refer to~\cite{FLZ, J.J1, J.JL, J.JLS, SLL}.

Suppose that $I$ is a unit square. Let $T$ be an isosceles right triangle and let $\Delta$ be an equilateral triangle. Song and Su~\cite{SS1} proved that any collection of positive homothetic copies of $T$ with the total area not smaller than $2$ permits a covering of $I$. Januszewski and Zielonka~\cite{JZ} generalized this result to rectangles. Song and Su~\cite{SS2} showed that any collection of positive homothetic copies of $\Delta$ with the total area not smaller than $1+\frac{7\sqrt{3}}{12}$ permits a covering of $I$. All these lower bounds are tight.

In this note we consider the problem of parallel packing a square with isosceles right triangles and equilateral triangles. In the following packing, both translations and $180^{\circ}$ rotations (denote by $\lambda_{n}$) are used. Denote by $|C|$ the area of the plane convex body $C$.
Let $\varrho(D, C)$ be the greatest number such that any collection of homothetic copies of $C$ with the total area not greater than $\varrho(D, C)\cdot|C|$ can be parallel packed into $D$. In Section~\ref{sec2} we show that $\varrho(I, T) = \frac{1}{2}$ and in Section~\ref{sec3} we prove that $\varrho(I, \Delta) = \frac{\sqrt{3}}{4}$.

\section{Packing a square with isosceles right triangles}\label{sec2}

\begin{lemma}\label{lem2.1}
	Suppose that $R$ is a rectangle with side lengths $a$ and $b$. Let $T$ be an isosceles right triangle with legs parallel to the sides of $R$. Moreover, let $\{T_{n}\}$ be a collection of homothetic copies of $T$ such that $t_{1}\geq t_{2}\geq\ldots$, where $t_{n}$ denotes the leg length of $T_{n}$ for $n=1,2,\ldots$. If
	\begin{align*}
		\sum|T_{n}|\leq\frac{1}{2}\cdot t^{2}_{1}+(a-t_{1})\cdot(b-t_{1})
	\end{align*}
	and if $t_{1}<\min\{a,b\}$, then the triangles from $\{T_{n}\}$ can be parallel packed into $R$.
\end{lemma}

\begin{proof}
	Let $R:=ABCD$. We describe a method of packing $T_{1}, T_{2}, \ldots$ into $R$. The method is based on the method from~\cite{J.W} and~\cite{J.J1}. The side $AB$ of $R$ is called the base. The sides $AD$ and $BC$ of $R$ are called the left and the right side, respectively. We pack the isosceles right triangles from the collection in layers, whose bases are parallel to the base of $R$. We pack the isosceles right triangles into $R$, beginning from the left side of $R$ along the base of the layer. The base of the first layer is the base of $R$. The assumption $t_{1}<\min\{a, b\}$ guarantees that $T_{1}$ can be packed into the first layer. If an isosceles right triangle cannot be packed into $R$ in a layer, then the new layer is created directly above the first isosceles right triangle lying in the preceding layer as in Fig.~\ref{f01} (each isosceles right triangle $\lambda_{n}T_{n}$ is denoted by $T_{n}$ for short). This packing process is terminated either as long as all isosceles right triangles have been packed or as long as there is an isosceles right triangle which cannot be packed into $R$ by this method.
	
	We are going to prove the lemma indirectly. Assume on the contrary that the isosceles right triangles from $\{T_{n}\}$ cannot be packed into $R$ by the method described in the preceding paragraph.
	
	\begin{figure}[!ht]
		\centering
		\includegraphics[width=8cm]{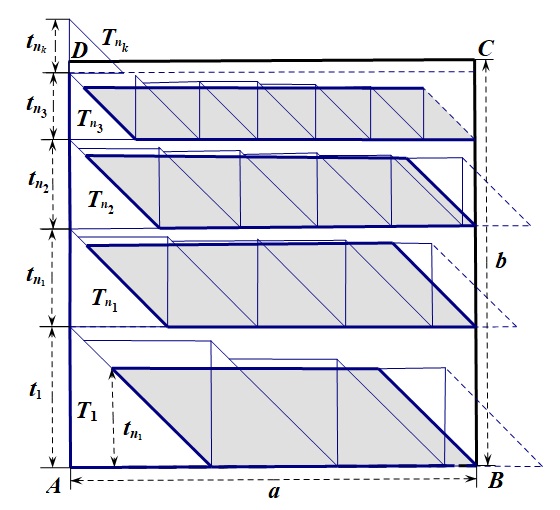}
		\caption{A packing method of the rectangle with side lengths $a$ and $b$}
		\label{f01}
	\end{figure}
	
	Denote by $k$ the number of created layers in which at least one isosceles right triangle is packed ($k = 4$ in Fig.~\ref{f01}). Denote by $T_{n_{j}}$
	the first isosceles right triangle lying in the $(j + 1)$-st layer for $j = 1, 2, \ldots, k-1$ (see Fig.~\ref{f01}, $n_{1}=7$ and $n_{2}=15$). Furthermore, let $T_{n_{k}}$
	be the isosceles right triangle which stops the packing process.
	
	We have
	\begin{align*}
		t_{1}+t_{n_{1}}+\cdots+t_{n_{k}}>b.
	\end{align*}
	
	Hence,
	\begin{equation}\label{equ1}
		t_{n_{1}}+\cdots+t_{n_{k}}>b-t_{1}.
	\end{equation}
	
	Observe that the sum $ \sum\limits_{n=2}^{n_{1}}|T_{n}|$ is greater than the area of the gray parallelogram with base length $a-t_{1}$ and with height $t_{n_{1}}$ as in Fig.~\ref{f01}. That is,
	\begin{align*}
		\sum\limits_{n=2}^{n_{1}}|T_{n}|>(a-t_{1})\cdot t_{n_{1}}.
	\end{align*}
	
	Similarly, from $t_{1}\geq t_{n_{1}} \geq t_{n_{2}}\geq\ldots\geq t_{n_{k-1}}$ we have
	
	\begin{align*}
		\sum\limits_{n=n_{1}+1}^{n_{2}}|T_{n}|>(a-t_{n_{1}})\cdot t_{n_{2}}\geq(a-t_{1})\cdot t_{n_{2}},
	\end{align*}
	
	\begin{align*}
		\sum\limits_{n=n_{2}+1}^{n_{3}}|T_{n}|>(a-t_{n_{2}})\cdot t_{n_{3}}\geq(a-t_{1})\cdot t_{n_{3}},
	\end{align*}
	
	\begin{align*}
		\ldots,
	\end{align*}
	
	\begin{align*}
		\sum\limits_{n=n_{k-1}+1}^{n_{k}}|T_{n}|>(a-t_{n_{k-1}})\cdot t_{n_{k}}\geq(a-t_{1})\cdot t_{n_{k}}.
	\end{align*}
	
	Therefore, by \eqref{equ1} we get
	\begin{align*}
		\sum\limits_{n=1}^{n_{k}}|T_{n}|&>\frac{1}{2}\cdot t^{2}_{1}+ (a-t_{1})\cdot (t_{n_{1}}+t_{n_{2}}+\cdots+t_{n_{k}})\\
		&>\frac{1}{2}\cdot t^{2}_{1}+ (a-t_{1})\cdot(b-t_{1}),
	\end{align*}
	which is a contradiction.
\end{proof}

\begin{figure}[!ht]
	\centering
	\includegraphics[width=8cm]{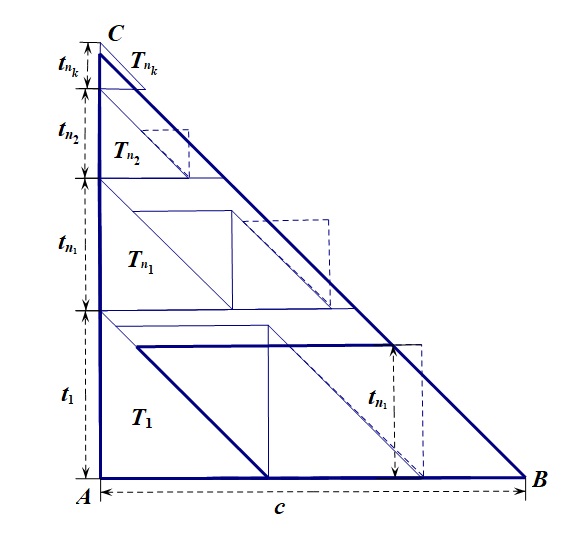}
	\caption{A packing method of the isosceles right triangle with leg length $c$}
	\label{f02}
\end{figure}

The following result can be found in~\cite{J.J0} and~\cite{RT}. Here we present an alternative method of proof.

\begin{lemma}\label{lem2.2}
	Suppose that $T^{*}$ is an isosceles right triangle with leg length $c$. Let $\{T_{n}\}$ be a collection of homothetic copies of $T^{*}$ such that $c>t_{1}\geq t_{2}\geq\ldots$, where $t_{n}$ denotes the leg length of $T_{n}$ for $n=1,2,\ldots$. If
	\begin{align*}
		\sum|T_{n}|\leq\frac{1}{2}\cdot t^{2}_{1}+\frac{1}{2}\cdot(c-t_{1})^{2},
	\end{align*}
	then the isosceles right triangles from $\{T_{n}\}$ can be parallel packed into $T^{*}$.
\end{lemma}

\begin{proof}
	The method of packing $T_{1}, T_{2}, \ldots$ into $T^{*}$ is similar to that of Lemma~\ref{lem2.1}.
	We are going to prove the lemma indirectly. Assume on the contrary that $\{T_{n}\}$ does not permit a packing into $T^{*}$.
	Denote by $k$ the number of created layers in which at least one isosceles right triangle is packed. Denote by $T_{n_{j}}$
	the first isosceles right triangle lying in the $(j + 1)$-th layer for $j =1, 2, \ldots, k-1$ (see Fig.~\ref{f02}, $n_{1}=4$ and $n_{2}=7$). Furthermore, let $T_{n_{k}}$ be the isosceles right triangle which stops the packing process.
	
	We have
	\begin{align*}
		t_{1}+t_{n_{1}}+\cdots+t_{n_{k}}>c.
	\end{align*}
	That is,
	\begin{equation}\label{equ2}
		t_{n_{1}}+\cdots+t_{n_{k}}>c-t_{1}.
	\end{equation}
	
	Observe that the sum $\sum\limits_{n=2}^{n_{1}}|T_{n}|$ is greater than the area of a parallelogram with base length $c-t_{1}$ and with height $t_{n_{1}}$ minus the area of an isosceles right triangle with leg length $t_{n_{1}}$ as in Fig.~\ref{f02}. That is,
	
	\begin{align*}
		\sum\limits_{n=2}^{n_{1}}|T_{n}|>(c-t_{1})\cdot t_{n_{1}}-\frac{1}{2}\cdot t^{2}_{n_{1}}.
	\end{align*}
	
	Similarly, we get
	
	\begin{align*}
		\sum\limits_{n=n_{1}+1}^{n_{2}}|T_{n}|>(c-t_{1}-t_{n_{1}})\cdot t_{n_{2}}-\frac{1}{2}\cdot t^{2}_{n_{2}},
	\end{align*}
	
	\begin{align*}
		\sum\limits_{n=n_{2}+1}^{n_{3}}|T_{n}|>(c-t_{1}-t_{n_{1}}-t_{n_{2}})\cdot t_{n_{3}}-\frac{1}{2}\cdot t^{2}_{n_{3}},
	\end{align*}
	
	\begin{align*}
		\ldots,
	\end{align*}
	
	\begin{align*}
		\sum\limits_{n=n_{k-1}+1}^{n_{k}}|T_{n}|>(c-t_{1}-t_{n_{1}}-t_{n_{2}}-\cdots-t_{n_{k-1}})\cdot t_{n_{k}}-\frac{1}{2}\cdot t^{2}_{n_{k}}.
	\end{align*}
	\begin{figure}[!ht]
		\centering
		\includegraphics[width=6cm]{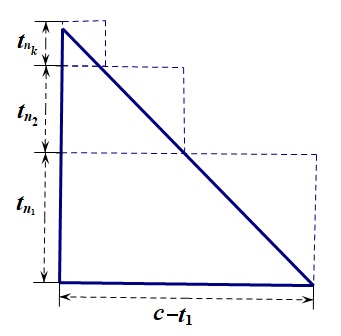}
		\caption{An isosceles right triangle with leg length $c-t_{1}$}
		\label{f03}
	\end{figure}
	By~\eqref{equ2} we see that
	\begin{align*}
		&(c-t_{1})\cdot t_{n_{1}}-\frac{1}{2}\cdot t^{2}_{n_{1}}\\
		+&(c-t_{1}-t_{n_{1}})\cdot t_{n_{2}}-\frac{1}{2}\cdot t^{2}_{n_{2}}\\
		+&(c-t_{1}-t_{n_{1}}-t_{n_{2}})\cdot t_{n_{3}}-\frac{1}{2}\cdot t^{2}_{n_{3}}\\
		+&\cdots\\
		+&(c-t_{1}-t_{n_{1}}-t_{n_{2}}-\cdots-t_{n_{k-1}})\cdot t_{n_{k}}-\frac{1}{2}\cdot t^{2}_{n_{k}}
	\end{align*}
	is greater than the area of the isosceles right triangle with leg length $c-t_{1}$, as shown in Fig.~\ref{f03}.
	
	Therefore,
	\begin{align*}\sum|T_{n}|>\frac{1}{2}\cdot t^{2}_{1}+\frac{1}{2}\cdot(c-t_{1})^{2},
	\end{align*}
	which is a contradiction.
\end{proof}

\begin{theorem}\label{thm.1}
	Suppose that $T$ is an isosceles right triangle with legs parallel to the sides of the unit square $I$. Let $\{T_{n}\}$ be a collection  of homothetic copies of $T$. If $\sum|T_{n}|\leq\frac{1}{2}$, then the triangles from $\{T_{n}\}$ can be parallel packed into $I$. Moreover, $\varrho(I,T)=\frac{1}{2}$.
\end{theorem}

\begin{proof}
	Suppose that $T$ is an isosceles right triangle with legs parallel to the sides of $I$ and that $\{T_{n}\}$ is a collection of homothetic copies of $T$.
	Let $\sum|T_{n}|\leq\frac{1}{2}$. Denote by $t_{n}$ the leg length of $T_{n}$ for $n=1,2,\ldots$. Without loss of generality we may assume that $t_{1}\geq t_{2} \geq \ldots$.
	
	{\it Case 1.}~$t_{1}+t_{3}< 1$.
	
	Place the isosceles right triangles $T_{1}$, $T_{2}$, $T_{3}$ and $T_{4}$ as in Fig.~\ref{f04}, the remaining isosceles right triangles $T_{5}, T_{6}, \ldots$ are used for packing the rectangle $R^{*}$ with side lengths $1-t_{2}$ and $1-(t_{1}-t_{2})$.
	
	\begin{figure}[!ht]
		\centering
		\includegraphics[width=6cm]{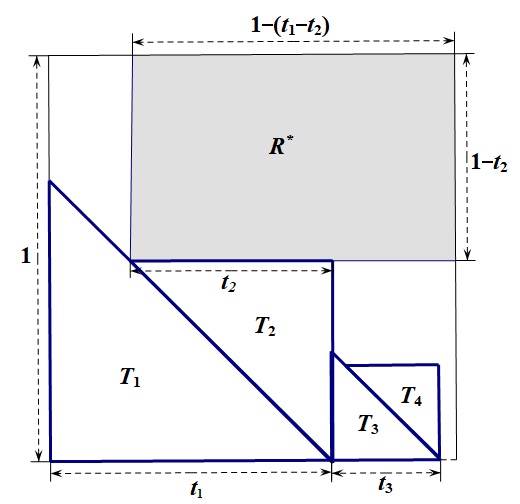}
		\caption{$t_{1}+t_{3}< 1$}
		\label{f04}
	\end{figure}
	
	By Lemma~\ref{lem2.1} we see that if the isosceles right triangles from $\{T_{n}\}$ cannot be parallel packed into $I$, then
	\begin{align*}
		\sum|T_{n}|&>\frac{1}{2}\cdot (t^{2}_{1}+t^{2}_{2}+t_{3}^{2}+t_{4}^{2})+\frac{1}{2}\cdot t^{2}_{5}+(1-t_{1}+t_{2}-t_{5})\cdot(1-t_{2}-t_{5})\\
		&\geq\frac{1}{2}\cdot (t^{2}_{1}+t^{2}_{2}+3\cdot t_{5}^{2})+(1-t_{1}+t_{2}-t_{5})\cdot(1-t_{2}-t_{5}).
	\end{align*}

	Using the standard method for finding the absolute minimum of a function of three variables, one checks that this lower bound is not less than $\frac{1}{2}$ provided $1\geq t_{1}\geq t_{2}\geq t_{5}>0$ and $t_{1}+t_{5}\leq t_{1}+t_{3}< 1$. The minimum value $\frac{1}{2}$ is attended for $t_{1}=t_{2}=t_{5}=\frac{1}{3}$, which is a contradiction.
	
	{\it Case 2.}~$t_{1}+t_{3}\geq 1$.
	
	Since $t_{1}+t_{2}\geq t_{1}+t_{3}\geq 1$, we place $T_{1}$ and $T_{2}$ as in Fig.~\ref{f05}, the remaining isosceles right triangles $T_{3}, T_{4}, \ldots$ are used for packing the isosceles right triangle $T^{c}$ with leg length $2-t_{1}-t_{2}$.
	
	By Lemma~\ref{lem2.2} we see that if the isosceles right triangles from $\{T_{n}\}$ cannot be parallel packed into $I$, then
	\begin{align*}
		\sum|T_{n}|>\frac{1}{2}\cdot (t^{2}_{1}+t^{2}_{2})+\frac{1}{2}\cdot t^{2}_{3}+\frac{1}{2}\cdot(2-t_{1}-t_{2}-t_{3})^{2}.
	\end{align*}
	
	A standard minimization argument shows that this lower bound is not less than $\frac{1}{2}$ provided $1\geq t_{1}\geq t_{2}\geq t_{3}>0$ and $t_{1}+t_{3}\geq 1$. The minimum value $\frac{1}{2}$ is attended for $t_{1}=t_{2}=t_{3}=\frac{1}{2}$, which is a contradiction.
	
	\begin{figure}[!ht]
		\centering
		\includegraphics[width=6cm]{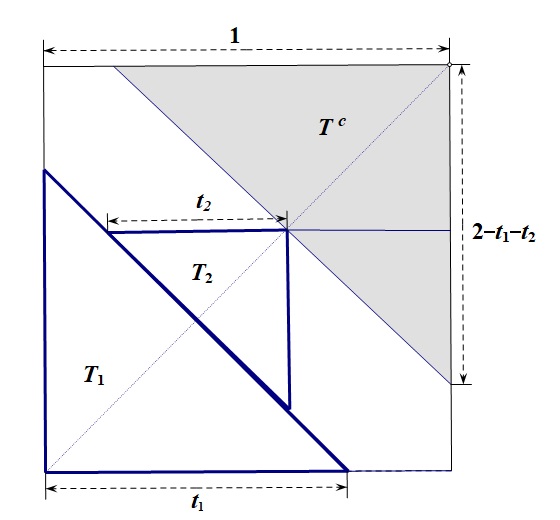}
		\caption{$t_{1}+t_{3}\geq 1$}
		\label{f05}
	\end{figure}
	
	Observe that one isosceles right triangle with leg length greater than $1$ (or four isosceles right triangles with leg lengths greater than $\frac{1}{2}$)
	cannot be parallel packed into $I$. Therefore, $\varrho(I,T)=\frac{1}{2}$.
\end{proof}

\begin{corollary}\label{thm.2}
	Suppose that $T$ is an isosceles right triangle with legs parallel to the diagonals of the unit square $I$. Let $\{T_{n}\}$ be a collection of homothetic copies of $T$. If $\sum|T_{n}|\leq\frac{1}{4}$, then the triangles from $\{T_{n}\}$ can be parallel packed into $I$. Moreover, $\varrho(I,T)=\frac{1}{4}$.
\end{corollary}

\begin{figure}[!ht]
	\centering
	\includegraphics[width=6cm]{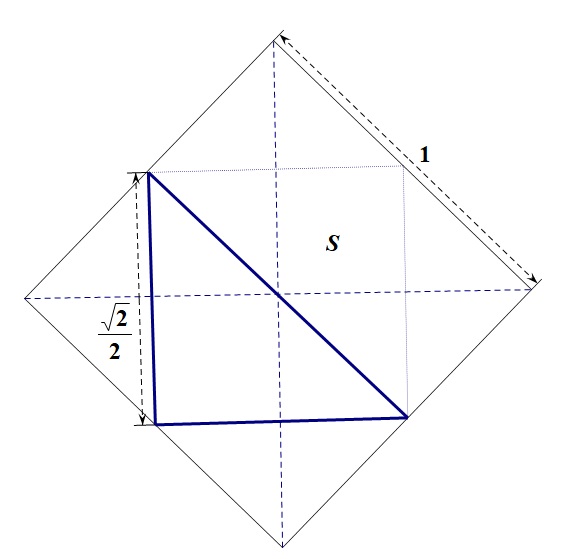}
	\caption{A unit square $I$ and the square $S$ with side length $\frac{\sqrt{2}}{2}$}
	\label{f06}
\end{figure}

\begin{proof}
	Joining the midpoints of the four sides of $I$, we obtain a square $S$ with side length $\frac{\sqrt{2}}{2}$ as in Fig.~\ref{f06}. By Theorem~\ref{thm.1} we know that if $\sum|T_{n}|\leq\frac{1}{2}|S|=\frac{1}{4}$, then the triangles from $\{T_{n}\}$ can be parallel packed into $S$ (and thus can be parallel packed into $I$). Moreover, any isosceles right triangle with leg length greater than $\frac{\sqrt{2}}{2}$ (and with legs parallel to the diagonals of the square $I$) cannot be packed into $I$ (see Fig.~\ref{f06} again). Therefore, $\varrho(I,T)=\frac{1}{4}$.
\end{proof}

\section{Packing a square with equilateral triangles}\label{sec3}

\begin{lemma}\label{lem3.1}
	Suppose that $R$ is a rectangle with side lengths $a$ and $h$. Let $\Delta$ be an equilateral triangle with a side parallel to the side of length $a$. Moreover, let $\{\Delta_{n}\}$ be a collection of homothetic copies of $\Delta$ such that $t_{1}\geq t_{2}\geq\ldots$, where $t_{n}$ denotes the side length of $\Delta_{n}$ for $n=1,2,\ldots$. If
	\begin{align*}
		\sum|\Delta_{n}|\leq\frac{\sqrt{3}}{4}\cdot t^{2}_{1}+(a-t_{1})\cdot(h-\frac{\sqrt{3}}{2} t_{1})
	\end{align*}
	and if $t_{1}<\min\{a,\frac{2\sqrt{3}}{3}h\}$, then the triangles from $\{\Delta_{n}\}$ can be parallel packed into $R$.
\end{lemma}

\begin{proof}
	Let $R:=ABCD$, where the length of $AB$ is $a$. We describe a method of packing $\Delta_{1}, \Delta_{2}, \ldots$ into $R$. The method is based on the method from~\cite{J.W} and~\cite{J.J1}. The side $AB$ of $R$ is called the base. The sides $AD$ and $BC$ of $R$ are called the left and the right side, respectively. We pack the equilateral triangles from the collection in layers, whose bases are parallel to the base of $R$. We pack the equilateral triangles into $R$, beginning from the left side of $R$ along the base of the layer. The base of the first layer is the base of $R$. The assumption $t_{1}<\min\{a, \frac{2\sqrt{3}}{3}h\}$ guarantees that $\Delta_{1}$ can be packed into the first layer. If an equilateral triangle cannot be packed into $R$ in a layer, then the new layer is created directly above the first equilateral triangle lying in the preceding layer as in Fig.~\ref{f07} (in all figures $\lambda_{n}\Delta_{n}$ is denoted by integer $n$, for short). This packing process is terminated either as long as all equilateral triangles have been packed or as long as there is an equilateral triangle which cannot be packed into $R$ by this method.
	
	We are going to prove the lemma indirectly. Assume on the contrary that the equilateral triangles from $\{\Delta_{n}\}$ cannot be packed into $R$ by the method described in the preceding paragraph.
	\begin{figure}[!ht]
		\centering
		\includegraphics[width=8cm]{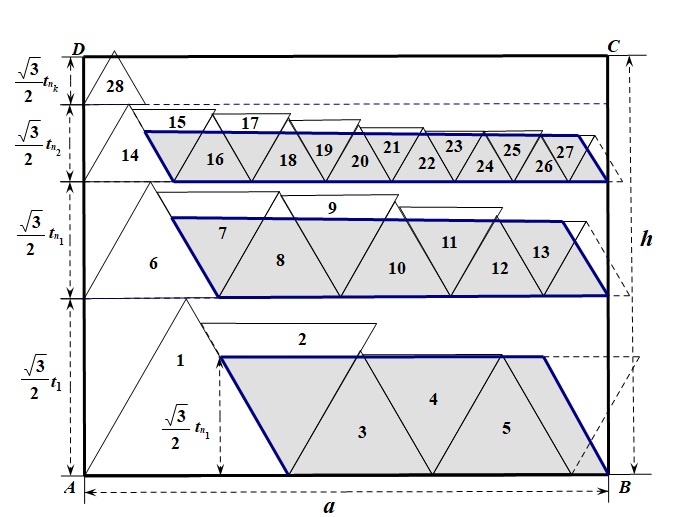}
		\caption{A packing method of the rectangle $R$.}
		\label{f07}
	\end{figure}
	Denote by $k$ the number of created layers in which at least one equilateral triangle is packed. Denote by $\Delta_{n_{j}}$
	the first equilateral triangle lying in the $(j + 1)$-st layer for $j = 1, 2, \ldots, k-1$ (in Fig.~\ref{f07}, $k = 3$, $n_{1}=6$ and $n_{2}=14$). Furthermore, let $\Delta_{n_{k}}$
	be the equilateral triangle which stops the packing process.
	
	We have
	\begin{align*}
		\frac{\sqrt{3}}{2}\cdot(t_{1}+t_{n_{1}}+\cdots+t_{n_{k}})>h.
	\end{align*}
	Consequently,
	\begin{equation}\label{equ3}
		\frac{\sqrt{3}}{2}\cdot(t_{n_{1}}+\cdots+t_{n_{k}})>h-\frac{\sqrt{3}}{2} t_{1}.
	\end{equation}
	
	Notice that the sum $\sum\limits_{n=2}^{n_{1}}|\Delta_{n}|$ is greater than the area of the gray parallelogram with base length $a-t_{1}$ and with height $\frac{\sqrt{3}}{2}t_{n_{1}}$ as in Fig.~\ref{f07}. That is,
	\begin{align*}
		\sum\limits_{n=2}^{n_{1}}|\Delta_{n}|>(a-t_{1})\cdot\frac{\sqrt{3}}{2} t_{n_{1}}.
	\end{align*}
	
	Similarly, from $t_{1}\geq t_{n_{1}}\geq t_{n_{2}}\geq\ldots\geq t_{n_{k-1}}$ we know that
	
	\begin{align*}
		\sum\limits_{n=n_{1}+1}^{n_{2}}|\Delta_{n}|>(a-t_{n_{1}})\cdot\frac{\sqrt{3}}{2} t_{n_{2}}\geq(a-t_{1})\cdot\frac{\sqrt{3}}{2} t_{n_{2}},
	\end{align*}
	
	\begin{align*}
		\sum\limits_{n=n_{2}+1}^{n_{3}}|\Delta_{n}|>(a-t_{n_{2}})\cdot\frac{\sqrt{3}}{2} t_{n_{3}}\geq(a-t_{1})\cdot\frac{\sqrt{3}}{2} t_{n_{3}},
	\end{align*}
	
	\begin{align*}
		\ldots,
	\end{align*}
	
	\begin{align*}
		\sum\limits_{n=n_{k-1}+1}^{n_{k}}|\Delta_{n}|>(a-t_{n_{k-1}})\cdot\frac{\sqrt{3}}{2} t_{n_{k}}\geq(a-t_{1})\cdot\frac{\sqrt{3}}{2} t_{n_{k}}.
	\end{align*}
	
	Therefore, by \eqref{equ3} we get
	\begin{align*}
		\sum\limits_{n=1}^{n_{k}}|\Delta_{n}|&>\frac{\sqrt{3}}{4}\cdot t^{2}_{1}+ (a-t_{1})\cdot\frac{\sqrt{3}}{2}\cdot (t_{n_{1}}+t_{n_{2}}+\cdots+t_{n_{k}})\\
		&>\frac{\sqrt{3}}{4}\cdot t^{2}_{1}+ (a-t_{1})\cdot(h-\frac{\sqrt{3}}{2} t_{1}),
	\end{align*}
	which is a contradiction.
\end{proof}

\begin{lemma}\label{lem3.2}
	Suppose that $Z$ is a right trapezoid with the shorter base of length $b$, with acute angle $60^{\circ}$ and with height $h$. Let $\Delta$ be an equilateral triangle with a side parallel to the bases of $Z$, and let $\{\Delta_{n}\}$ be a collection of homothetic copies of $\Delta$ such that $t_{1}\geq t_{2}\geq\ldots$, where $t_{n}$ denotes the side length of $\Delta_{n}$ for $n=1,2,\ldots$. If
	\begin{align*}
		\sum|\Delta_{n}|\leq\frac{\sqrt{3}}{4}\cdot t^{2}_{1}+(b-\frac{3}{4}t_{1}+\frac{\sqrt{3}}{6}h)\cdot(h-\frac{\sqrt{3}}{2}t_{1}),
	\end{align*}
	then the equilateral triangles from $\{\Delta_{n}\}$ can be parallel packed into $Z$.
\end{lemma}

\begin{proof}
	The method of packing $\Delta_{1}, \Delta_{2}, \ldots$ into $Z$ is similar to that of Lemma~\ref{lem3.1}.
	We are going to prove the lemma indirectly. Assume on the contrary that the equilateral triangles from $\{\Delta_{n}\}$ cannot be parallel packed into $Z$.
	Denote by $k$ the number of created layers in which at least one equilateral triangle is packed. Denote by $\Delta_{n_{j}}$
	the first equilateral triangle lying in the $(j + 1)$-th layer for $j =1, 2, \ldots, k-1$ (see Fig.~\ref{f08}, $n_{1}=5$ and $n_{2}=10$). Furthermore, let $\Delta_{n_{k}}$ be the equilateral triangle which stops the packing process.
	\begin{figure}[!ht]
		\centering
		\includegraphics[width=8cm]{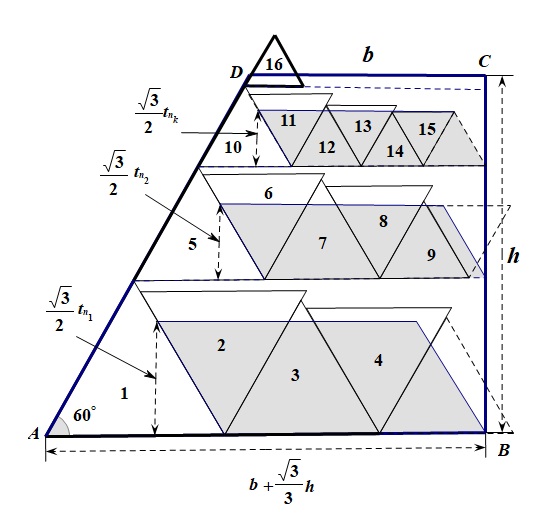}
		\caption{A packing method of the right trapezoid $Z$.}
		\label{f08}
	\end{figure}
	
	We have
	\begin{align*}
		\frac{\sqrt{3}}{2}(t_{1}+t_{n_{1}}+\cdots+t_{n_{k}})>h.
	\end{align*}
	That is,
	\begin{equation}\label{equ4}
		\frac{\sqrt{3}}{2}(t_{n_{1}}+\cdots+t_{n_{k}})>h-\frac{\sqrt{3}}{2}t_{1}.
	\end{equation}
	
	Observe that the sum $\sum\limits_{n=2}^{n_{1}}|\Delta_{n}|$ is greater than the area of the gray parallelogram with base length $b+\frac{\sqrt{3}}{3}h-t_{1}$ and with height $\frac{\sqrt{3}}{2} t_{n_{1}}$ as in Fig.~\ref{f08}. That is,
	
	\begin{align*}
		\sum\limits_{n=2}^{n_{1}}|\Delta_{n}|>(b+\frac{\sqrt{3}}{3}h-t_{1})\cdot\frac{\sqrt{3}}{2} t_{n_{1}}.
	\end{align*}
	
	Similarly, from $t_{1}\geq t_{n_{1}}\geq t_{n_{2}}\geq\ldots\geq t_{n_{k-1}}$ we know that
	
	\begin{align*}
		\sum\limits_{n=n_{1}+1}^{n_{2}}|\Delta_{n}|&>(b+\frac{\sqrt{3}}{3}h-\frac{1}{2}t_{1}-t_{n_{1}})\cdot\frac{\sqrt{3}}{2} t_{n_{2}}\\
		&\geq(b+\frac{\sqrt{3}}{3}h-t_{1}-\frac{1}{2}t_{n_{1}})\cdot\frac{\sqrt{3}}{2} t_{n_{2}},
	\end{align*}
	
	\begin{align*}
		\sum\limits_{n=n_{2}+1}^{n_{3}}|\Delta_{n}|&>(b+\frac{\sqrt{3}}{3}h-\frac{1}{2}t_{1}-\frac{1}{2}t_{n_{1}}-t_{n_{2}})\cdot\frac{\sqrt{3}}{2} t_{n_{3}}\\
		&\geq(b+\frac{\sqrt{3}}{3}h-t_{1}-\frac{1}{2}t_{n_{1}}-\frac{1}{2}t_{n_{2}})\cdot\frac{\sqrt{3}}{2} t_{n_{3}},
	\end{align*}
	
	\begin{align*}
		\ldots,
	\end{align*}
	
	\begin{align*}
		\sum\limits_{n=n_{k-1}+1}^{n_{k}}|\Delta_{n}|&>(b+\frac{\sqrt{3}}{3}h-\frac{1}{2}t_{1}-\frac{1}{2}t_{n_{1}}-\frac{1}{2}t_{n_{2}}-\cdots-t_{n_{k-1}})\cdot\frac{\sqrt{3}}{2} t_{n_{k}}\\
		&\geq(b+\frac{\sqrt{3}}{3}h-t_{1}-\frac{1}{2}t_{n_{1}}-\frac{1}{2}t_{n_{2}}-\cdots-\frac{1}{2}t_{n_{k-1}})\cdot\frac{\sqrt{3}}{2} t_{n_{k}}.
	\end{align*}
	
	\begin{figure}[!ht]
		\centering
		\includegraphics[width=6cm]{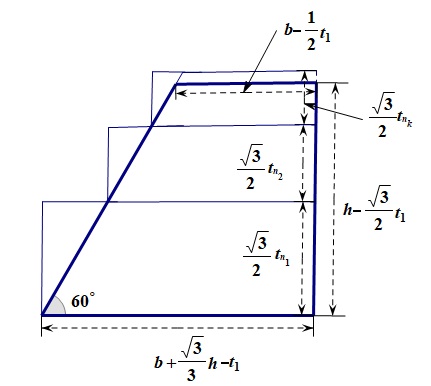}
		\caption{A right trapezoid with bases $b-\frac{1}{2}t_{1}$ and $b+\frac{\sqrt{3}}{3}h-t_{1}$, and with height $h-\frac{\sqrt{3}}{2}t_{1}$.}
		\label{f09}
	\end{figure}
	
	As shown in Fig.~\ref{f09}, by~\eqref{equ4} we see that the sum
	\begin{align*}
		&(b+\frac{\sqrt{3}}{3}h-t_{1})\cdot\frac{\sqrt{3}}{2} t_{n_{1}}\\
		+&(b+\frac{\sqrt{3}}{3}h-t_{1}-\frac{1}{2}t_{n_{1}})\cdot\frac{\sqrt{3}}{2} t_{n_{2}}\\
		+&(b+\frac{\sqrt{3}}{3}h-t_{1}-\frac{1}{2}t_{n_{1}}-\frac{1}{2}t_{n_{2}})\cdot\frac{\sqrt{3}}{2} t_{n_{3}}\\
		+&\cdots\\
		+&(b+\frac{\sqrt{3}}{3}h-t_{1}-\frac{1}{2}t_{n_{1}}-\frac{1}{2}t_{n_{2}}-\cdots-\frac{1}{2}t_{n_{k-1}})\cdot\frac{\sqrt{3}}{2} t_{n_{k}}
	\end{align*}
	is greater than the area of the right trapezoid with base lengths $b-\frac{1}{2}t_{1}$ and $b+\frac{\sqrt{3}}{3}h-t_{1}$, and with height $h-\frac{\sqrt{3}}{2}t_{1}$, that is, $(b-\frac{3}{4}t_{1}+\frac{\sqrt{3}}{6}h)\cdot(h-\frac{\sqrt{3}}{2}t_{1})$.
	
	Therefore,
	\begin{align*}
		\sum|\Delta_{n}|>\frac{\sqrt{3}}{4}\cdot t^{2}_{1}+(b-\frac{3}{4}t_{1}+\frac{\sqrt{3}}{6}h)\cdot(h-\frac{\sqrt{3}}{2}t_{1}),
	\end{align*}
	which is a contradiction.
\end{proof}

\begin{theorem}
	Let $\Delta$ be an equilateral triangle with a side parallel to a side of $I$, and let $\{\Delta_{n}\}$ be a collection of homothetic copies of $\Delta$. If $\sum|\Delta_{n}|\leq \frac{\sqrt{3}}{4}$, then the equilateral triangles from $\{\Delta_{n}\}$ can be parallel
	packed into $I$. Moreover, $\varrho(I,\Delta)=\frac{\sqrt{3}}{4}$.
\end{theorem}
\begin{figure}[!ht]
	\centering
	\includegraphics[width=8cm]{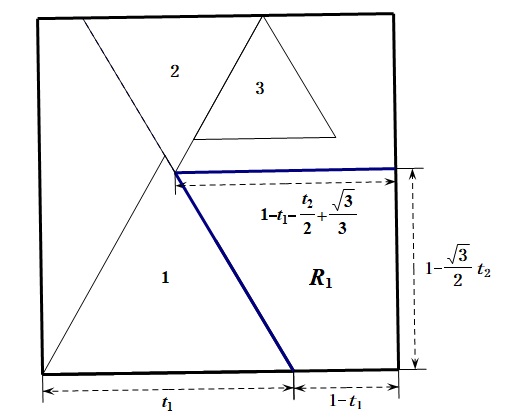}
	\caption{$t_{1}+t_{3}\geq 1$.}
	\label{f010}
\end{figure}
\begin{proof}
	Let $\sum|\Delta_{n}|\leq\frac{\sqrt{3}}{4}$. Denote by $t_{n}$ the side length of $\Delta_{n}$ for $n=1,2,\ldots$. Without loss of generality, we may assume that $t_{1}\geq t_{2} \geq \ldots$.
	
	{\it Case 1.}~$\frac{\sqrt{3}}{3}\leq t_{1}< 1$.
	
	{\it Subcase 1.1.}~$t_{1}+\frac{1}{2}t_{2}< 1$.
	
	{\it Subacse 1.1.1.}~$t_{1}+t_{3}\geq 1$.

	Place the equilateral triangles $\Delta_{1}$, $\Delta_{2}$ and $\Delta_{3}$ as in Fig.~\ref{f010}, the remaining equilateral triangles $\Delta_{4}, \Delta_{5}, \ldots$ are used for packing the right trapezoid $R_{1}$ with base lengths $1-t_{1}$ and $1-t_{1}-\frac{1}{2}t_{2}+\frac{\sqrt{3}}{3}$, and with height $1-\frac{\sqrt{3}}{2}t_{2}$.
	
	By Lemma~\ref{lem3.2} we see that if the equilateral triangles from $\{\Delta_{n}\}$ cannot be parallel packed into $I$, then
	\begin{align*}
		\sum|\Delta_{n}|>\frac{\sqrt{3}}{4}\cdot (t^{2}_{1}+t^{2}_{2}+t_{3}^{2}+t_{4}^{2})+(1-t_{1}-\frac{1}{4}t_{2}-\frac{3}{4}t_{4}+\frac{\sqrt{3}}{6})\cdot(1-\frac{\sqrt{3}}{2}t_{2}-\frac{\sqrt{3}}{2}t_{4}).
	\end{align*}
	
	From $t_{1}+\frac{1}{2}t_{2}<1$, $t_{1}+t_{3}\geq 1$, $t_{1}\geq t_{2}\geq t_{3}\geq t_{4}$ and $\frac{\sqrt{3}}{3}\leq t_{1}< 1$ we know that the lower bound
	is not less than $\frac{\sqrt{3}}{3}-\frac{1}{8}$ (the minimum value is attended for $t_{1}=\frac{3}{4}-\frac{\sqrt{3}}{12}$ and $t_{2}=t_{3}=t_{4}=\frac{1}{4}+\frac{\sqrt{3}}{12}$), which is greater than $\frac{\sqrt{3}}{4}$, a contradiction.
	
	{\it Subcase 1.1.2.}~$t_{1}+t_{3}< 1$.
	
	Place $\Delta_{1}$, $\Delta_{2}$ and $\Delta_{3}$ as in Fig.~\ref{f011}, the remaining equilateral triangles $\Delta_{4}, \Delta_{5}, \ldots$ are used for packing the right trapezoid $R_{2}$ with base lengths $1-t_{1}+\frac{1}{2}t_{2}$ and $1-t_{1}+\frac{\sqrt{3}}{3}$, and with height $1-\frac{\sqrt{3}}{2}t_{2}$.
	
	By Lemma~\ref{lem3.2} we see that if the equilateral triangles from $\{\Delta_{n}\}$ cannot be parallel packed into $I$, then
	\begin{align*}
		\sum|\Delta_{n}|>\frac{\sqrt{3}}{4}\cdot (t^{2}_{1}+t^{2}_{2}+t_{3}^{2}+t_{4}^{2})+(1-t_{1}+\frac{1}{4}t_{2}-\frac{3}{4}t_{4}+\frac{\sqrt{3}}{6})\cdot(1-\frac{\sqrt{3}}{2}t_{2}-\frac{\sqrt{3}}{2}t_{4}).
	\end{align*}
	
	\begin{figure}[!ht]
		\centering
		\includegraphics[width=8cm]{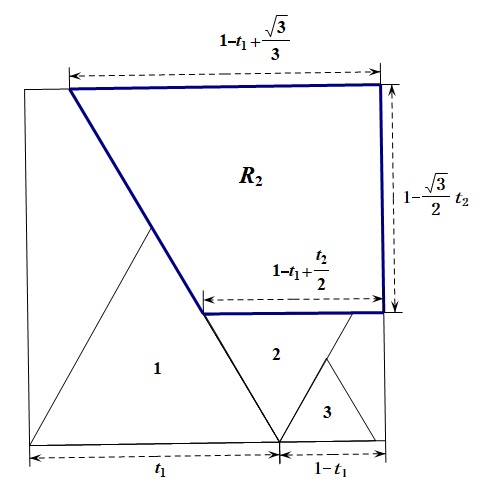}
		\caption{$t_{1}+t_{3}< 1$.}
		\label{f011}
	\end{figure}
	
	From $t_{1}+\frac{1}{2}t_{2}<1$, $t_{1}+t_{3}< 1$, $t_{1}\geq t_{2}\geq t_{3}\geq t_{4}$ and $\frac{\sqrt{3}}{3}\leq t_{1}< 1$ we know that the lower bound
	is not less than $2-\frac{31\sqrt{3}}{36}$ (the minimum value is attended for $t_{1}=\frac{\sqrt{3}}{3}$ and $t_{2}=2-\frac{8\sqrt{3}}{9}$ and $t_{3}=t_{4}=\frac{2\sqrt{3}}{9}$), which is greater than $\frac{\sqrt{3}}{4}$, a contradiction.
	
	{\it Subcase 1.2.}~$t_{1}+\frac{1}{2}t_{2}\geq 1$.
	
	{\it Subcase 1.2.1.~The triangle $\Delta_{2}$ cannot be parallel packed into $I$}.
	
	\begin{figure}[!ht]
		\centering
		\includegraphics[width=6cm]{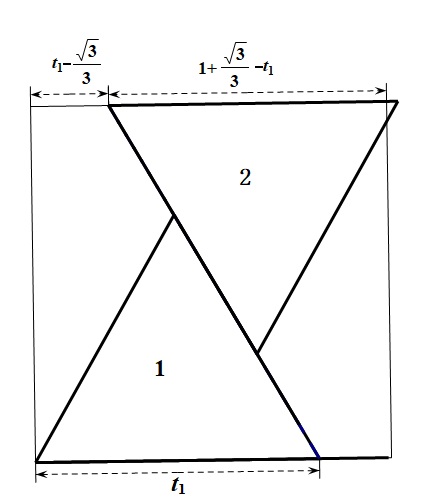}
		\caption{$\Delta_{2}$ cannot be packed into $I$.}
		\label{f012}
	\end{figure}
	
	If $\Delta_{2}$ cannot be parallel packed into $I$, then $t_{2}>1+\frac{\sqrt{3}}{3}-t_{1}$ as in Fig.~\ref{f012}. Therefore,
	
	\begin{align*}
		|\Delta_{1}|+|\Delta_{2}|&=\frac{\sqrt{3}}{4}(t_{1}^{2}+t_{2}^{2})>\frac{\sqrt{3}}{4}(t_{1}^{2}+(1+\frac{\sqrt{3}}{3}-t_{1})^2)\\
		&\geq\frac{1}{4}+\frac{\sqrt{3}}{6}>\frac{\sqrt{3}}{4},
	\end{align*}
	which is a contradiction.
	
	\begin{figure}[!ht]
		\centering
		\includegraphics[width=6cm]{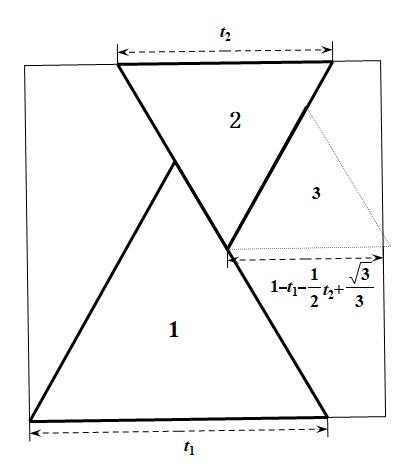}
		\caption{$\Delta_{3}$ cannot be packed into $I$.}
		\label{f013}
	\end{figure}
	
	{\it Subcase 1.2.2.~The triangle $\Delta_{2}$ can be parallel packed into $I$}.
	
	{\it Subcase 1.2.2.1.~The triangle $\Delta_{3}$ cannot be parallel packed into $I$}.
	
	If $\Delta_{2}$ can be parallel packed into $I$ while $\Delta_{3}$ cannot, then $t_{3}>1+\frac{\sqrt{3}}{3}-t_{1}-\frac{1}{2}t_{2}$ as in Fig.~\ref{f013}.
	
	As a consequence,
	\begin{align*}
		|\Delta_{1}|+|\Delta_{2}|+|\Delta_{3}|=\frac{\sqrt{3}}{4}(t_{1}^{2}+t_{2}^{2}+t_{3}^{2})>\frac{\sqrt{3}}{4}t_{1}^{2}
		+\frac{\sqrt{3}}{4}t_{2}^{2}+\frac{\sqrt{3}}{4}(1+\frac{\sqrt{3}}{3}-t_{1}-\frac{1}{2}t_{2})^{2}>\frac{\sqrt{3}}{4},
	\end{align*}
	which is a contradiction.
	
	{\it Subcase 1.2.2.2.~The triangle $\Delta_{3}$ can be parallel packed into $I$}.
	
	Place $\Delta_{1}$, $\Delta_{2}$ and $\Delta_{3}$ as in Fig.~\ref{f014}, the remaining equilateral triangles $\Delta_{4}, \Delta_{5}, \ldots $ are used for packing the right trapezoid $R_{3}$ with base lengths $1-t_{1}$ and $1+\frac{\sqrt{3}}{3}-t_{1}-\frac{1}{2}t_{2}$, and with height $1-\frac{\sqrt{3}}{2}t_{2}$.
	
	\begin{figure}[!ht]
		\centering
		\includegraphics[width=8cm]{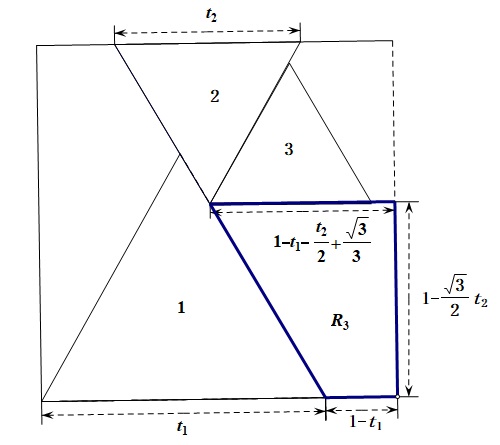}
		\caption{$\Delta_{3}$ can be packed into $I$.}
		\label{f014}
	\end{figure}
	
	By Lemma~\ref{lem3.2} we see that if the equilateral triangles from $\{\Delta_{n}\}$ cannot be parallel packed into $I$, then
	\begin{align*}
		\sum|\Delta_{n}|>\frac{\sqrt{3}}{4}\cdot (t^{2}_{1}+t^{2}_{2}+t_{3}^{2}+t_{4}^{2})+(1-t_{1}-\frac{1}{4}t_{2}-\frac{3}{4}t_{4}+\frac{\sqrt{3}}{6})\cdot(1-\frac{\sqrt{3}}{2}t_{2}-\frac{\sqrt{3}}{2}t_{4}).
	\end{align*}
	
	From $t_{1}+\frac{1}{2}t_{2}\geq 1$, $t_{1}\geq t_{2}\geq t_{3}\geq t_{4}$ and $\frac{\sqrt{3}}{3}\leq t_{1}< 1$ we know that the lower bound
	is not less than $\frac{2}{19}+\frac{25\sqrt{3}}{114}$ (the minimum value is attended for $t_{1}=\frac{36+4\sqrt{3}}{57}$, $t_{2}=\frac{42-8\sqrt{3}}{57}$ and $t_{3}=t_{4}=\frac{4\sqrt{3}-2}{19}$), which is greater than $\frac{\sqrt{3}}{4}$, a contradiction.
	
	{\it Case 2.}~$\frac{2}{5}\leq t_{1}<\frac{\sqrt{3}}{3}$.
	
	{\it Subcase 2.1.}~$t_{1}+t_{3}< 1$.
	
	Place $\Delta_{1}$, $\Delta_{2}$ and $\Delta_{3}$ as in Fig.~\ref{f015}, the remaining equilateral triangles $\Delta_{4}, \Delta_{5}, \ldots $ are used for packing the right trapezoid $R_{4}$ with base lengths $1$ and $1-\frac{\sqrt{3}}{3}+\frac{1}{2}t_{2}$, and with height $1-\frac{\sqrt{3}}{2}t_{2}$.
	
	\begin{figure}[!ht]
		\centering
		\includegraphics[width=8cm]{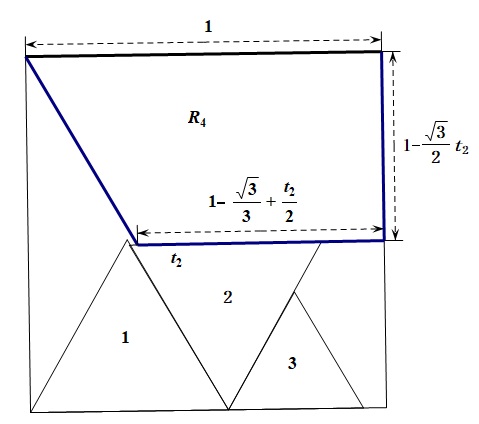}
		\caption{$t_{1}+t_{3}< 1$.}
		\label{f015}
	\end{figure}
	
	By Lemma~\ref{lem3.2} we see that if the equilateral triangles from $\{\Delta_{n}\}$ cannot be parallel packed into $I$, then
	\begin{align*}
		\sum|\Delta_{n}|>\frac{\sqrt{3}}{4}\cdot (t^{2}_{1}+t^{2}_{2}+t_{3}^{2}+t_{4}^{2})+(1+\frac{1}{4}t_{2}-\frac{3}{4}t_{4}-\frac{\sqrt{3}}{6})\cdot(1-\frac{\sqrt{3}}{2}t_{2}-\frac{\sqrt{3}}{2}t_{4}).
	\end{align*}
	
	From $t_{1}+t_{3}<1$, $t_{1}\geq t_{2}\geq t_{3}\geq t_{4}$ and $\frac{2}{5}\leq t_{1}<\frac{\sqrt{3}}{3}$ we know that the lower bound
	is not less than $\frac{38}{35}-\frac{79\sqrt{3}}{210}$ (the minimum value is attended for $t_{1}=t_{2}=\frac{2}{5}$ and $t_{3}=t_{4}=\frac{8}{35}+\frac{2\sqrt{3}}{21}$), which is greater than $\frac{\sqrt{3}}{4}$, a contradiction.
	
	{\it Subcase 2.2.}~$t_{1}+t_{3}\geq 1$.
	
	{\it Subcase 2.2.1.}~$t_{3}+t_{5}\geq 1$.
	
	Since $t_{1}+t_{3}> 1$ and $t_{3}+t_{5}> 1$, from $t_{1}\geq t_{2}\geq t_{3}\geq t_{4}\geq t_{5}$ we have
	\begin{align*}
		\sum|\Delta_{n}|&>\frac{\sqrt{3}}{4}\cdot (t^{2}_{1}+t^{2}_{2}+t_{3}^{2}+t_{4}^{2}+t_{5}^{2})
		\geq\frac{\sqrt{3}}{4}\cdot (3\cdot t_{3}^{2}+2\cdot t_{5}^{2})\\
		&>\frac{\sqrt{3}}{4}(3\cdot t_{3}^{2}+2(1-t_{3})^2)\geq\frac{3\sqrt{3}}{10}>\frac{\sqrt{3}}{4},
	\end{align*}
	which is a contradiction.
	
	{\it Subcase 2.2.2.}~$t_{3}+t_{5}< 1$.
	
	{\it Subcase 2.2.2.1.}~$t_{1}+t_{4}<1$.
	
	Place $\Delta_{1}$, $\Delta_{2}$, $\Delta_{3}$ and $\Delta_{4}$ as in Fig.~\ref{f016}, the remaining equilateral triangles $\Delta_{5}, \Delta_{6}, \ldots $ are used for packing the right trapezoid $R_{5}$ with base lengths $1-t_{3}$ and $1-t_{3}+\frac{\sqrt{3}}{3}-\frac{1}{2}t_{1}$, and with height $1-\frac{\sqrt{3}}{2}t_{1}$.
	
	\begin{figure}[!ht]
		\centering
		\includegraphics[width=8cm]{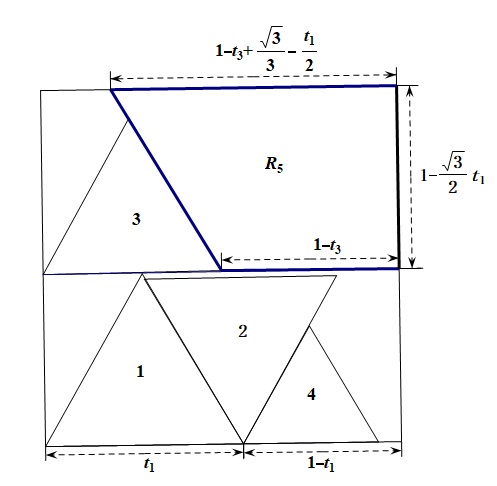}
		\caption{$t_{1}+t_{4}< 1$.}
		\label{f016}
	\end{figure}
	
	By Lemma~\ref{lem3.2} we see that if the equilateral triangles from $\{\Delta_{n}\}$ cannot be parallel packed into $I$, then
	\begin{align*}
		\sum|\Delta_{n}|&>\frac{\sqrt{3}}{4}\cdot (t^{2}_{1}+t^{2}_{2}+t_{3}^{2}+t_{4}^{2}+t_{5}^{2})+(1-t_{3}
		-\frac{1}{4}t_{1}-\frac{3}{4}t_{5}+\frac{\sqrt{3}}{6})\cdot(1-\frac{\sqrt{3}}{2}t_{1}-\frac{\sqrt{3}}{2}t_{5}).
	\end{align*}
	
	From $t_{1}+t_{3}\geq 1$, $t_{1}+t_{4}<1$, $t_{1}\geq t_{2}\geq t_{3}\geq t_{4}$ and $\frac{2}{5}\leq t_{1}<\frac{\sqrt{3}}{3}$ we know that the lower bound
	is not less than $\frac{145\sqrt{3}}{168}-1$ (the minimum value is attended for $t_{1}=\frac{\sqrt{3}}{3}$, $t_{2}=t_{3}=1-\frac{\sqrt{3}}{3}$ and $t_{4}=t_{5}=\frac{4\sqrt{3}}{21}$), which is greater than $\frac{\sqrt{3}}{4}$, a contradiction.
	
	{\it Subcase 2.2.2.2.}~$t_{1}+t_{4}\geq 1$.
	
	Place $\Delta_{1}$, $\Delta_{2}$, $\Delta_{3}$ and $\Delta_{4}$ as in Fig.~\ref{f017}, the remaining equilateral triangles $\Delta_{5}, \Delta_{6}, \ldots $ are used for packing the right trapezoid $R_{6}$ with base lengths $1-t_{3}$ and $1-t_{3}-\frac{\sqrt{3}}{3}+\frac{1}{2}t_{1}$, and with height $1-\frac{\sqrt{3}}{2}t_{1}$.
	
	\begin{figure}[!ht]
		\centering
		\includegraphics[width=8cm]{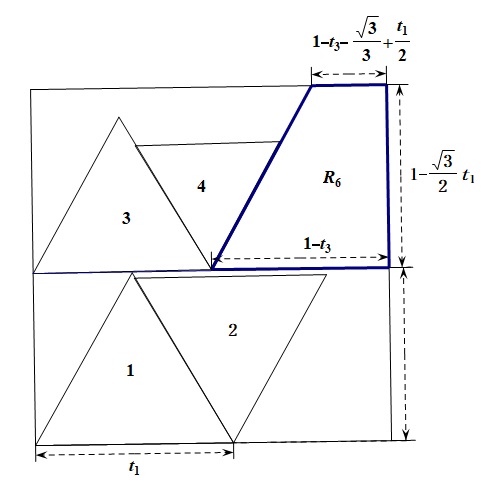}
		\caption{$t_{1}+t_{4}\geq 1$.}
		\label{f017}
	\end{figure}
	
	By Lemma~\ref{lem3.2} we see that if the equilateral triangles from $\{\Delta_{n}\}$ cannot be parallel packed into $I$, then
	\begin{align*}
		\sum|\Delta_{n}|&>\frac{\sqrt{3}}{4}\cdot (t^{2}_{1}+t^{2}_{2}+t_{3}^{2}+t_{4}^{2}+t_{5}^{2})+(1-t_{3}
		+\frac{1}{4}t_{1}-\frac{3}{4}t_{5}-\frac{\sqrt{3}}{6})\cdot(1-\frac{\sqrt{3}}{2}t_{1}-\frac{\sqrt{3}}{2}t_{5}).
	\end{align*}
	
	From $t_{3}+t_{5}< 1$, $t_{1}+t_{4}\geq 1$, $t_{1}\geq t_{2}\geq t_{3}\geq t_{4}$ and $\frac{2}{5}\leq t_{1}<\frac{\sqrt{3}}{3}$ we know that the lower bound
	is not less than $\frac{17\sqrt{3}}{15}-\frac{3}{2}$ (the minimum value is attended for $t_{1}=\frac{\sqrt{3}}{3}$, $t_{2}=t_{3}=t_{4}=1-\frac{\sqrt{3}}{3}$ and $t_{5}=\frac{\sqrt{3}}{5}$), which is greater than $\frac{\sqrt{3}}{4}$, a contradiction.
	
	{\it Case 3.}~$0<t_{1}<\frac{2}{5}$.
	
	By Lemma~\ref{lem3.1} we see that if the equilateral triangles from $\{\Delta_{n}\}$ cannot be parallel packed into $I$, then from $0<t_{1}<\frac{2}{5}$ we get
	
	\begin{align*}
		\sum|\Delta_{n}|>\frac{\sqrt{3}}{4}t_{1}^{2}+(1-t_{1})\cdot(1-\frac{\sqrt{3}}{2}t_{1})\geq\frac{3}{5}-\frac{2\sqrt{3}}{25}>\frac{\sqrt{3}}{4},
	\end{align*}
	which is a contradiction.
	
	Observe that one equilateral triangle with side length greater than $1$ cannot be parallel packed into $I$. As a consequence,
	\begin{align*}
		\varrho(I, \Delta)=\frac{\sqrt{3}}{4}.
	\end{align*}
	
\end{proof}

%\bibliography{sn-bibliography}

\end{document}